\documentclass[notitlepage,11pt,reqno]{amsart}
\usepackage[foot]{amsaddr}
\usepackage{amssymb,nicefrac,bm,upgreek,mathtools,verbatim}
\usepackage[final]{hyperref}
\usepackage[mathscr]{eucal}
\usepackage{dsfont}
\usepackage[normalem]{ulem}
\usepackage{amsopn,esint}
\usepackage[T1]{fontenc}

\usepackage[utf8]{inputenc}
\usepackage{apptools}
\AtAppendix{\counterwithin{lemma}{section}} 

\usepackage[margin=.95in]{geometry}
\newcommand{\stkout}[1]{\ifmmode\text{\sout{\ensuremath{#1}}}\else\sout{#1}\fi}

\newtheorem{lemma}{Lemma}[section]
\newtheorem{theorem}{Theorem}[section]
\newtheorem{proposition}{Proposition}[section]

\theoremstyle{definition}
\newtheorem{definition}{Definition}[section]

\theoremstyle{remark}
\newtheorem{remark}{Remark}[section]
\numberwithin{theorem}{section}
\numberwithin{equation}{section}

\hypersetup{
  colorlinks=true,
  citecolor=mblue,
  linkcolor=mblue,
  urlcolor = blue,
  anchorcolor = blue,
  frenchlinks=false,
  pdfborder={0 0 0},
  naturalnames=false,
  hypertexnames=false,
  breaklinks}
\usepackage[capitalize,nameinlink]{cleveref}

\usepackage[abbrev,msc-links,nobysame,citation-order]{amsrefs}

\crefname{section}{Section}{Sections}
\crefname{subsection}{Section}{Sections}
\crefname{condition}{Condition}{Conditions}
\crefname{hypothesis}{Hypothesis}{Conditions}
\crefname{assumption}{Assumption}{Assumptions}
\crefname{lemma}{Lemma}{Lemmas}
\crefname{fact}{Fact}{Facts}

\Crefname{figure}{Figure}{Figures}

\crefformat{equation}{\textup{#2(#1)#3}}
\crefrangeformat{equation}{\textup{#3(#1)#4--#5(#2)#6}}
\crefmultiformat{equation}{\textup{#2(#1)#3}}{ and \textup{#2(#1)#3}}
{, \textup{#2(#1)#3}}{, and \textup{#2(#1)#3}}
\crefrangemultiformat{equation}{\textup{#3(#1)#4--#5(#2)#6}}%
{ and \textup{#3(#1)#4--#5(#2)#6}}{, \textup{#3(#1)#4--#5(#2)#6}}%
{, and \textup{#3(#1)#4--#5(#2)#6}}

\Crefformat{equation}{#2Equation~\textup{(#1)}#3}
\Crefrangeformat{equation}{Equations~\textup{#3(#1)#4--#5(#2)#6}}
\Crefmultiformat{equation}{Equations~\textup{#2(#1)#3}}{ and \textup{#2(#1)#3}}
{, \textup{#2(#1)#3}}{, and \textup{#2(#1)#3}}
\Crefrangemultiformat{equation}{Equations~\textup{#3(#1)#4--#5(#2)#6}}%
{ and \textup{#3(#1)#4--#5(#2)#6}}{, \textup{#3(#1)#4--#5(#2)#6}}%
{, and \textup{#3(#1)#4--#5(#2)#6}}

\crefdefaultlabelformat{#2\textup{#1}#3}
\newcommand{\vertiii}[1]{{\left\vert\kern-0.25ex\left\vert\kern-0.25ex\left\vert #1
    \right\vert\kern-0.25ex\right\vert\kern-0.25ex\right\vert}}

\usepackage{color}
\definecolor{dmagenta}{rgb}{.4,.1,.5}
\definecolor{dblue}{rgb}{.0,.0,.5}
\definecolor{mblue}{rgb}{.0,.0,.7}
\definecolor{ddblue}{rgb}{.0,.0,.4}
\definecolor{dred}{rgb}{.7,.0,.0}
\definecolor{dgreen}{rgb}{.0,.5,.0}
\definecolor{Eeom}{rgb}{.0,.0,.5}

\allowdisplaybreaks

\newcommand{\ttl}{A NOTE ON HOPF'S LEMMA AND STRONG MINIMUM PRINCIPLE FOR NONLOCAL EQUATIONS WITH NON-STANDARD GROWTH}
\begin{document}
\title[Hopf lemma and minimum principle]
{\ttl}

\author{Abhrojyoti Sen}

\address{Department of Mathematics, Indian Institute of Science Education and Research, Dr. Homi Bhabha Road,
Pune 411008, India. Email: abhrojyoti.sen@acads.iiserpune.ac.in }

\begin{abstract}
Let $\Omega\subset \mathbb{R}^n $ be any open set and $u$ be a weak supersolution of $\mathcal{L}u=c(x)g(|u|)\frac{u}{|u|}$
where 
\[\mathcal{L}u(x)=\text{p.v.} \int_{\mathbb{R}^n} g\left(\frac{|u(x)-u(y)|}{|x-y|^s}\right) \frac{u(x)-u(y)}{|u(x)-u(y)|} K(x,y)\frac{dy}{|x-y|^s}\] and $g=G^{\prime}$ for some Young function $G.$ This note imparts a Hopf's type lemma and strong minimum principle for $u$ when $c(x)$ is continuous in $\overline{\Omega}$ that extend the results of Del Pezzo and Quaas (JDE-2017) in fractional Orlicz-Sobolev setting.
\end{abstract}

\maketitle


\section{Introduction and main results}
This short note establishes a Hopf's boundary point lemma and strong minimum principle for the fractional nonlocal equation
\begin{equation} \label{nonlocal equation}
    \mathcal{L}u=c(x)g(|u|) \frac{u}{|u|} \,\,\,\, \, \text{in} \,\,\,\, \Omega,
\end{equation}
where $\Omega \subset \mathbb{R}^n$ is an open set and for $0<s<1$ the nonlocal operator $\mathcal{L}u$ is given by 
\begin{equation}\label{nonlocal operator}
    \mathcal{L}u(x):= \text{p.v.} \int_{\mathbb{R}^n} g\left(\frac{|u(x)-u(y)|}{|x-y|^s}\right) \frac{u(x)-u(y)}{|u(x)-u(y)|} \frac{K(x,y)}{|x-y|^s}dy.
\end{equation}
The function $g:[0, \infty) \to [0,\infty)$ is continuous, strictly increasing and $g$ satisfies $g(0)=0, \displaystyle{\lim_{t\to \infty}} g(t)=\infty.$ Furthermore, $g$ satisfies the following condition:
\begin{equation}\label{condition on g}
    1 < p_g \leq \frac{t g^{\prime}(t)}{g(t)} \leq q _g< \infty,\,\,\, \text{for some}\,\,\, 1<p_g \leq q_g.
\end{equation}
Also, we set the following assumptions on the kernel $K(x, y).$ $K : \mathbb{R}^n \times \mathbb{R}^n \to (0, \infty]$ satisfies
\begin{itemize}
    \item Symmetry : $K(x,y)=K(y,x)$ \,\,\,\, \text{for all}\,\,\, $x,y \in \mathbb{R}^n.$
    \item Translation invariance: $K(x+z, y+z)=K(x,y)$ \,\,\,\, \text{for all}\,\,\, $x,y, z \in \mathbb{R}^n.$
    \item Growth condition: $\frac{\lambda}{|x-y|^n} \leq K(x,y) \leq \frac{\Lambda}{|x-y|^n}$ \,\,\,\, \text{for all}\,\,\, $x,y \in \mathbb{R}^n$\,\,\, \text{and for some}\,\,\, $0<\lambda \leq \Lambda.$ 
    \item Continuity: $K(x,y)$ is continuous in both of its variable.
\end{itemize}
Note that when $K(x,y)=\frac{1}{|x-y|^n},$ the operator becomes the so called fractional $g$-Laplacian and it is denoted by $(-\Delta)^s_g.$ The well known fractional Laplacian can be obtained by simply taking $g:=Id.$ In recent years, these operators have received a lot of attention due to their applicability in various fields, for example, optimization, finance, phase transitions, thin obstacle problem, optimal transport, image reconstruction, Le\'{v}y processes, game theory, conservation laws, minimal surfaces, materials science, water waves,  as well as diffusion problems ( see \cite{Bertoin, CA, VA} and the references therein ). The structure of $(-\Delta)_g^s$ naturally suggests to consider {\it fractional Orlicz-Sobolev spaces} introduced by Bonder and Salort in \cite{bonder2}. Interested readers can see \cite{AC2021, AC2, BOT} and the references therein, where some structural properties of fractional Orlicz-Sobolev spaces have been studied recently. Coming back to the operator $(-\Delta)_g^s,$ several topics from the point of view of PDE, such as eigenvalue related problems \cite{SA1, SA2, SA3}, Liouville type theorem and symmetry results \cite{SA4}, regularity results \cite{salortNA2022, byun2022} have been addressed very recently. 

In this note, we are interested to establish a Hopf lemma and strong minimum principle for the equation \eqref{nonlocal equation} involving a general operator \eqref{nonlocal operator}. Hopf boundary point lemma is a classical and one of the most useful results in the theory of elliptic partial differential equations. This result was first proved by E. Hopf in his seminal paper \cite{H1} for second order elliptic PDE without the term $c(x)u.$ After this influential work of Hopf, there are various improvements and applications came into the literature, for instance, it is observed that Hopf lemma is useful to validate strong maximum principle for second order uniformly elliptic operators, see \cite{pucci1, pucci2} for a comprehensive overview. Moreover, it has applications to the symmetry results and overdetermined problems for local and nonlocal elliptic equations \cite{BJ20, FJ15, FK08, GS2016}. In \cite{BL21}, Hopf lemma, maximum principle and certain symmetry results for a general class of nonlocal operators have been recently established.

There is a large class of functions $g$ that satisfies \eqref{condition on g}. To cite some of them we mention, $g(t)=Ct^{p-1}, p>1$ and some constant $C>\frac{1}{p-1},$ $g(t)=t^{\alpha}\log(\beta + \gamma t)$ where $\alpha >1, \beta \geq e^{M_0}$ for some $M_0>0$ and $\gamma \geq 0.$ Furthermore, if $g_1,g_2$ are two functions satisfying \eqref{condition on g}, then their linear combination $\alpha g_1+\beta g_2$ with $\alpha, \beta \geq 0,$ product $g_1\cdot g_2$ and composition $g_1\circ g_2$ also satisfy \eqref{condition on g}. When $g(t)=Ct^{p-1}$ with $p>1,$ the operator \eqref{nonlocal operator} becomes the fractional $p$-Laplacian. Harnack inequality and other regularity results were established for fractional $p$-Laplace equations and nonlocal equations with general growth in \cite{DKP1,DKP2, FZ22}. Del Pezzo and Quaas \cite{delpezzoJDE2017} provided a Hopf lemma and strong minimum principle for the equation of type \eqref{nonlocal equation} when $\mathcal{L}$ is fractional $p$-Laplacian. Hopf lemma and strong maximum principles are obtained for other fractional $p$-Laplace equations and their generalizations in \cite{Li1, Li2, Li3, Li4, Li5, SZ22}. Furthermore, very recently similar kind of study has been carried out for regional fractional Laplacian \cite{ABF2022} and for fractional $(p,q)$-Laplacian operators \cite{VA2022}.
\subsection{Main results} In this section we state our main results. The first result is the strong minimum principle.
\begin{theorem}\label{mainth1}
Let $c\in L^1_{loc}(\Omega)$ (resp. $c\in C(\overline{\Omega})$) with $c(x) \leq 0$ in $\Omega$ and $u \in W^{s, G}_{loc} (\Omega) \cap L_s^g(\mathbb{R}^n)$ (resp. $u \in W^{s, G}_{loc} (\Omega) \cap L_s^g(\mathbb{R}^n) \cap C(\overline{\Omega})$) be a weak supersolution to 
\begin{align*}
    \mathcal{L}u(x)=c(x)g(|u|)\frac{u}{|u|} \,\,\, \text{in}\,\,\, \Omega.
\end{align*}
Then the following conclusions hold:
\begin{enumerate}
    \item For a bounded domain $\Omega,$ if $u$ satisfies the condition that $u\geq 0\,\,\, a.e.$ in $\Omega^c$ then either $u>0\,\,\, a.e.$ in $\Omega$ (resp. $u>0$ in $\Omega$) or $u=0 \,\,\, a.e.$ in $\mathbb{R}^n.$
    \item If  $u\geq 0\,\,\, a.e.$ in $\mathbb{R}^n,$ we have either $u>0\,\,\, a.e.$ in $\Omega$ (resp. $u>0$ in $\Omega$) or $u=0 \,\,\, a.e.$ in $\mathbb{R}^n.$
\end{enumerate}
\end{theorem}
Now we state our Hopf's boundary point lemma.
\begin{theorem}\label{mainth2}
Suppose $\Omega$ satisfies an interior ball condition at a point $y \in \partial \Omega,$ $c\in C(\overline{\Omega})$ and $u \in W^{s, G}_{loc} (\Omega) \cap L_s^g(\mathbb{R}^n) \cap C(\overline{\Omega})$ be a weak supersolution to
\begin{align*}
    \mathcal{L}u(x)=c(x)g(|u|)\frac{u}{|u|} \,\,\, \text{in}\,\,\, \Omega.
\end{align*}
Then the following conclusions hold:
\begin{enumerate}
    \item Let $\Omega$ be bounded, $c(x) \leq 0$ in $\Omega$ and let $u$ satisfy the condition that $u\geq 0\,\,\, a.e.$ in $\Omega^c.$ Also, denote $B_r(x_0)$ as a ball of radius $r$ centred at $x_0 \in \Omega$ that touches $y\in \partial \Omega$ from the interior and $\delta(x):=dist (x, B^c_r(x_0)).$ Then either 
    \begin{align}\label{Hlemma}
        \liminf_{B_r(x_0) \ni x \to y } \frac{u(x)}{\delta^s(x)} >0
    \end{align}
    
    or $u=0 \,\,\, a.e.$ in $\mathbb{R}^n.$
    \item Whenever $u\geq 0\,\,\, a.e.$ in $\mathbb{R}^n,$ we have either \eqref{Hlemma} holds or $u=0 \,\,\, a.e.$ in $\mathbb{R}^n.$
\end{enumerate}
\end{theorem}
The rest of the article is organized as follows. In \cref{sec2}, we introduce the necessary preliminaries and collect all auxiliary results which will be used throughout the article. In \cref{sec3} we prove \cref{mainth1}. Lastly, \cref{mainth2} is proved in \cref{sec4}.
\section{Preliminaries and Auxiliary lemmas}\label{sec2} This section collects all the necessary definitions, notions of solutions, and lemmas. We start with the {\it Young function.}
\subsection{Young functions} A $C^2$-function $G:[0, \infty) \to [0, \infty)$ is called a {\it Young function} if it is increasing, convex, $G(0)=0,$ 
\[\lim_{t \to 0}\frac{G(t)}{t}=0 ,\,\,\,\,\, \lim_{t\to \infty}\frac{G(t)}{t}=\infty\] and $G$ satisfies an integral representation 
\[G(t):=\int_0^t g(\theta)d\theta\] for some $g:[0,\infty) \to [0,\infty)$ satisfying the property \eqref{condition on g}. We extend $g$ in $\mathbb{R}$ by defining $g(t)=-g(-t)$ for $t<0.$ Simply by integrating \eqref{condition on g} and using the fundamental theorem of calculus, we obtain
\begin{align}\label{condition on G}
2<p \leq \frac{tg(t)}{G(t)} \leq q <\infty\,\,\,\, \textnormal{where}\,\,\, p=p_g+1, \,\,q=q_g+1.
\end{align}
For a Young function $G$ we define its conjugate $G^*:[0,\infty) \to [0,\infty)$ by 
\begin{align}\label{conjugate of G}
G^*(t):=\sup_{s\geq 0}\{st-G(s)\},\,\,\, t\geq 0.
\end{align} Moreover $G$ satisfies the following \textit{$\Delta_2$ condition}:
 $G(2t) \leq 2^q G(t)$ and using \eqref{condition on G}, this certainly implies $g(2t) \leq 2^{q_g}g(t).$ Further, \eqref{condition on g} implies
\begin{align*}
a^qG(t)\leq G(at) \leq a^p G(t)\,\,\, \text{for}\,\,\, 0<a<1\,\,\, \text{and}\,\,\, a^pG(t)\leq G(at) \leq a^q G(t)\,\,\, \text{for}\,\,\, a>1.
\end{align*} More generally, for any $a, t \geq 0,$ we have (see \cite[Lemma~2.1]{bonder})
 \begin{align*} 
 \min\left\{a^{p_g}, a^{q_g}\right\}g(t)\leq g(at) \leq \max\left\{a^{p_g}, a^{q_g}\right\}g(t)
 \end{align*}
 and
\begin{align*}
\min\left\{a^{p}, a^{q}\right\}G(t)\leq G(at) \leq \max\left\{a^{p}, a^{q}\right\} G(t).
\end{align*}
Furthermore, we impose the following condition on $g$:\\\\
  \textit{\hspace*{5cm}$g$ is a convex function on $[0, \infty).$ } \\\\
 This condition corresponds to the degenerate case $p\geq 2$ for fractional $p$-Laplacian. Notice that fine boundary regularity estimates are only available for the case $p\geq 2$ for fractional $p$-Laplace equation \cite{IMS}.
\subsection{Fractional Orlicz-Sobolev spaces} For any open set $\Omega \subset \mathbb{R}^n,$ we consider the class of all real valued measurable functions on $\Omega,$ denoted by $\mathcal{M}(\Omega).$ Given a Young function $G$ satisfying $g=G^{\prime}$ and $\Delta_2$ condition, we define the {\it Orlicz spaces} by
\[L^{G}(\Omega):=\left\{u\in \mathcal{M}(\Omega)\,\,\Big|\,\, \int_{\Omega}G\left(|u(x)|\right)dx < \infty\right\}.\] This forms a Banach space with the Luxemburg norm:
\[||u||_{L^{G}(\Omega)}:=\inf\left\{\lambda>0 \,\,\Big| \int_{\Omega}G\left(\frac{\left|u(x)\right|}{\lambda}\right)dx \leq 1\right\}.\]
For $s\in (0,1),$ we define the {\it fractional Orlicz-Sobolev space} as
\[W^{s,G}:=\left\{u \in L^G(\Omega)\,\, \Big|\,\, \int_{\Omega} \int_{\Omega}G\left(\frac{|u(x)-u(y)|}{|x-y|^s}\right)\frac{dxdy}{|x-y|^n} < \infty\right\}.\] This also forms a Banach space with the Luxemburg type norm
\[||u||_{W^{s,G}(\Omega)}:=||u||_{L^G(\Omega)}+[u]_{s, G, \Omega}\,\, ,\] where $[u]_{s, G, \Omega}$ is the Gagliardo semi-norm defined as
\[[u]_{s,G, \Omega}:=\inf\left\{\lambda>0 \,\, \Big|\,\, \int_{\Omega} \int_{\Omega} G\left(\frac{|u(x)-u(y)|}{\lambda |x-y|^s}\right)\frac{dxdy}{|x-y|^n}\leq 1\right\}.\] Further we define the spaces
\begin{align*}
W^{s, G}_{loc}(\Omega):=\left\{u \in L^G_{loc}(\mathbb{R}^n)\,\, \Big|\,\, u\in W^{s,G}(U)\right\},
\end{align*}
where $U$ be any open set such that $\Omega^{\prime} \subset \subset U$ for any bounded $\Omega^{\prime} \subset \Omega,$ and
\[W^{s,G}_0(\Omega):=\left\{u \in W^{s,G}(\mathbb{R}^n)\,\, \Big|\,\, u=0 \,\,\, \text{in }\,\,\, \mathbb{R}^n\setminus \Omega\right\},\]
\[W^{-s, G^*}(\Omega):=\left(W^{s,G}_0(\Omega)\right)^*,\] where $G^*$ denotes the conjugate of $G$ defined in \eqref{conjugate of G}. Next, we define
\begin{align*}
    L^g_s(\mathbb{R}^n):=\left\{u \in \mathcal{M}(\mathbb{R}^n) \,\, \Big|\,\, \int_{\mathbb{R}^n}g\left(\frac{|u(x)|}{(1+|x|)^s}\right)\frac{dx}{(1+|x|)^{n+s}} < \infty\right\},
\end{align*}
and the nonlocal tail of $u \in L^g_s(\mathbb{R}^n)$ for the ball $B_R(x_0)$ is define by
\begin{align*}
    \text{Tail}(u, x_0, R):= \int_{B^c_R(x_0)} g\left(\frac{|u(x)|}{|x-x_0|^s}\right)\frac{dx}{|x-x_0|^{n+s}}.
\end{align*}
\begin{remark}
Note that $u \in L^g_s(\mathbb{R}^n)$ if and only if $\text{Tail}(u, x_0, R) <\infty$ for all $x_0 \in \mathbb{R}^n$ and $R>0.$ Also if $u\in L^{\infty}(\mathbb{R}^n)$ then clearly $u \in L^g_s(\mathbb{R}^n).$
\end{remark}
\subsection{Notion of weak and viscosity solutions}In this section, we recall the notion of weak and viscosity solutions. Let us start with the weak solution.
\begin{definition}
Let $\Omega \subset \mathbb{R}^n$ be a bounded domain. A function $u \in W^{s,G}(\Omega)\cap L^g_s(\mathbb{R}^n)$ is said to be a weak supersolution (subsolution) of $\mathcal{L}u=f$ if
\begin{align}\label{weak solution}
\int_{\mathbb{R}^n}\int_{\mathbb{R}^n}g\left(\frac{|u(x)-u(y)|}{|x-y|^s}\right)\frac{u(x)-u(y)}{|u(x)-u(y)|} \left(\eta(x)-\eta(y)\right)K(x,y)\frac{dxdy}{|x-y|^s}  \geq  (\leq ) \int_{\Omega} f(x,u)\eta(x)dx
\end{align}
holds for all non-negative $\eta\in C^{\infty}_c(\Omega),$ where $f\in W^{-s, G^*}(\Omega).$

We say $u$ be a weak solution of $\mathcal{L}u=f$ if it is both sub and supersolution.

If $\Omega$ is unbounded, then $u \in W^{s, G}_{loc}(\Omega)\cap L^g_s(\mathbb{R}^n)$ is a weak supersolution (subsolution) if \eqref{weak solution} holds for any bounded $\Omega^{\prime} \subset \Omega.$
\end{definition}
Next, we recall the definition of viscosity solution from \cite[Definition~2.1]{delpezzoADE}. For any open bounded domain $\Omega \subset \mathbb{R}^n$ and $\beta \geq 2,$ we denote 
\[C^2_{\beta}(\Omega):=\left\{u \in C^2(\Omega) : \sup_{x\in \Omega}\left(\frac{\min \{\delta_u(x), 1\}^{\beta-1}}{\left|Du(x)\right|}+\frac{\left|D^2u(x)\right|}{\delta^{\beta-2}_u(x)}\right) < \infty\right\},\]
where $\delta_u(x):=dist(x, N_u)$ and $N_u:=\{x\in \Omega: Du(x)=0\}$ denotes the set of all critical points of $u.$ Further, denote $u_{+}:=\max\{u, 0\}$ and $u_{-}:=\max\{-u, 0\}.$ Now we introduce our definition below.
\begin{definition}\label{defn2.2}
We say that a function $u : \mathbb{R}^n \to [-\infty, \infty]$ is a viscosity supersolution (subsolution) of \eqref{nonlocal equation} in $\Omega$ if it satisfies the following conditions:\\
(1) $u < +\infty$ ($u> -\infty$) $a.e.$ in $\mathbb{R}^n$ and $u>-\infty$ ($u< \infty$) $a.e.$ in $\Omega.$\\
(2) $u$ is lower (upper) semicontinuous in $\Omega.$\\
(3) If $\psi \in C^{2}(B_r(x_0))\cap L^g_s(\mathbb{R}^n)$ for some $B_r(x_0) \subset \Omega$ such that $\psi(x_0)=u(x_0)$ and $\psi \leq u$ ($\psi \geq u$) in $B_r(x_0)$ and one of the following holds:
\begin{enumerate}
    \item [(i)] $p > \frac{2}{2-s}$ or $D \psi(x_0) \neq 0,$
    \item [(ii)] $1 <p \leq \frac{2}{2-s}, D\psi(x_0)=0$ such that $x_0$ is an isolated point of $\psi$ and $\psi \in C^2_{\beta}(B_r(x_0)) $ for some $\beta > \frac{sp}{p-1},$ 
\end{enumerate}
where $p$ is defined in \eqref{condition on G},
then
\begin{align}\label{visinq}
\mathcal{L}\psi_r(x_0)\geq (\leq)\,\, c(x_0)g(|u(x_0)|)\frac{u(x_0)}{|u(x_0)|},
\end{align}
where 
\begin{align*}
\psi_r(x)=\begin{cases}
\psi(x)\,\,\,\, \textnormal{in}\,\,\, B_r(x_0)\\
u(x)\,\,\,\, \textnormal{otherwise}.
\end{cases}
\end{align*}
(4)  $u_{-} (u_{+}) \in L^g_s(\mathbb{R}^n).$

A function $u$ is said to be a viscosity solution to \eqref{nonlocal equation} if it is both viscosity sub and  supersolution.
\end{definition}
\begin{remark}
One basic difference between the definitions given in \cite{delpezzoADE} and \cite{erikJMPA} is following: \cite{delpezzoADE} needed a larger set of test functions $\psi$ such that $\psi \leq u$ in $\mathbb{R}^n,$ since the non-homogeneous term depends on $$D^s_gu(x):=\int_{\mathbb{R}^n}G\left(\frac{|u(x)-u(y)|}{|x-y|^s}\right)K(x,y)dy$$ as well. As pointed out in \cite[Remark 2.3]{medina2021}, if the non-homogeneous term depends only on $x$ and $u$ then one can allow $\psi_r$ as a test function where $\psi\leq u$ only in $B_r(x_0).$ In our case using $\psi_r$ in \eqref{visinq} is not restrictive.
\end{remark}
\subsection{Auxiliary lemmas}
In this section, we collect all the results that are needed to prove \cref{mainth1} and \cref{mainth2}. With some cosmetic change of \cite[Proposition~C.4, Lemma~C.5, Proposition~4.5]{salortNA2022} we can get the following results:
\begin{lemma}\label{salortNAlemma}
Let $u\in W^{s,G}_{loc}(\Omega) \cap L^g_s(\mathbb{R}^n)$ be a weak solution of $\mathcal{L}u=f$ where $f \in L^1_{loc}(\Omega).$ Let $v \in L^1_{loc}(\mathbb{R}^n)$ satisfies,

    \[dist\left(supp(v), \Omega\right)>0\,\,\,\, \text{and}\,\,\,\, \int_{\mathbb{R}^n\setminus \Omega}g\left(\frac{|u(x)|}{(1+|x|)^s}\right)\frac{dx}{(1+|x|)^{n+s}}< \infty. \] For $a.e.$ Lebesgue points of $u,$ define
 \begin{align*}
h(x):=2 \int_{supp(v)}\Big[&g\left(\frac{\left|u(x)-u(y)-v(y)\right|}{\left|x-y\right|^s}\right)\frac{u(x)-u(y)-v(y)}{\left|u(x)-u(y)-v(y)\right|}\\
&-g\left(\frac{\left|u(x)-u(y)\right|}{\left|x-y\right|^s}\right)\frac{u(x)-u(y)}{\left|u(x)-u(y)\right|}\Big] K(x,y)\frac{dy}{\left|x-y\right|^s}
\end{align*}
Then $u+v \in W^{s,G}_{loc}(\Omega) \cap L^g_s(\mathbb{R}^n)$ and is a weak solution of $\mathcal{L}(u+v)=f+h$ in $\Omega.$
\end{lemma}
\begin{proposition}\label{prop2.1}
Let $\Omega$ be a bounded domain and $u, v \in W^{s, G}(\Omega)\cap L^g_s(\mathbb{R}^n)$ such that $u\leq v$ in $\Omega^c.$ If 
\begin{align*}
&\int_{\mathbb{R}^n} \int_{\mathbb{R}^n} g\left(\frac{|u(x)-u(y)|}{|x-y|^s}\right) \frac{u(x)-u(y)}{|u(x)-u(y)|} \frac{(\eta(x)-\eta(y))K(x,y)}{|x-y|^s}dxdy\\
&\leq \int_{\mathbb{R}^n} \int_{\mathbb{R}^n} g\left(\frac{|v(x)-v(y)|}{|x-y|^s}\right) \frac{v(x)-v(y)}{|v(x)-v(y)|} \frac{(\eta(x)-\eta(y))K(x,y)}{|x-y|^s}dxdy
\end{align*}
holds for all $\eta \in W^{s, G}_0,$ $\eta \geq 0,$ then $u\leq v$ in $\Omega.$
\end{proposition}
\begin{remark}
Comparison principle also holds for non-negative subsolution and supersolution of \eqref{nonlocal equation} with $c(x) \leq 0.$
\end{remark}
\begin{proposition}\label{prop2.2}
Let $\Omega \subset \mathbb{R}^n$ be a bounded domain with $C^{1,1}$ boundary and $\delta(x):=dist(x, \Omega^c).$ Then there exists a $\rho>0$ depending on $s, \Omega$ such that
\[\mathcal{L}\delta^s=f\,\,\,\, weakly\,\,\,\,in\,\,\,\, \Omega_{\rho} \] for some $f \in L^{\infty}(\Omega_{\rho})$ where $\Omega_{\rho}:=\left\{x\in \Omega | \delta(x) < \rho\right\}.$
\end{proposition}
Further, we state the results which can be obtained by slightly modifying \cite[Lemma~3.8, Lemma~3.9]{delpezzoADE}. 
\begin{lemma}\label{regular peturbation}
Let $\Omega \subset \mathbb{R}^n$ be a bounded domain, $B_r(x_0) \subset \Omega$ and $\psi \in C^2(B_r(x_0))\cap L^{\infty}(\mathbb{R}^n)$ satisfying Definition 2.4, 3 (i) or (ii) with $\beta > \frac{sp}{p-1}.$ Then for all $\epsilon >0$ and $\tilde{\rho} >0,$ there exist $\tilde{\theta}>0, \rho\in (0, \tilde{\rho})$ and $\eta \in C^2_c(B_{\frac{\rho}{2}}(x_0))$ with $0\leq \eta \leq 1$ and $\eta(x_0)=1 $ such that $\psi_{\theta}(x)=\psi(x)+\theta \eta(x)$ satisfies 
\begin{align*}
    \sup_{B_{\rho}(x_0)}\left|\mathcal{L}\psi-\mathcal{L}\psi_{\theta}\right| < \epsilon\,\,\,\, \text{for all}\,\,\,\, 0\leq \theta < \tilde{\theta}.
\end{align*}
\end{lemma}
\begin{lemma}\label{Lemma 2.3}
Let $\Omega \subset \mathbb{R}^n$ be a bounded domain, $B_r(x_0) \subset \Omega$ and $\psi \in C^2(B_r(x_0))\cap L^g_s(\mathbb{R}^n)$. Furthermore, assume that $\psi \in C^2_{\beta}(B_r(x_0))$ for some $\beta > \frac{sp}{p-1}$ when $1<p\leq \frac{2}{2-s}$ and $D\psi(x_0)=0$ where $x_0$ is an isolated point in $B_r(x_0).$ Then $\mathcal{L}\psi$ is continuous in $B_r(x_0).$
\end{lemma}
We prove the next lemma in the same spirit of \cite{fractional e-values}. 
\begin{lemma}\label{lemma2.4}
Let $\Omega$ be a bounded domain and $\Omega^{\prime} \subset \Omega.$ We also assume that the following inequality
\begin{align}\label{inq2.6}
    \mathcal{L}u(x) \leq f(x)\,\,\, \text{for all}\,\,\, x\in \Omega^{\prime}
\end{align}
holds for $u\in C^1_0(\mathbb{R}^n),$ $f\in C(\Omega).$ Then we have
\begin{align}\label{inq3}
    \int_{\mathbb{R}^n} \int_{\mathbb{R}^n} g\left(\frac{|u(x)-u(y)|}{|x-y|^s}\right) \frac{u(x)-u(y)}{|u(x)-u(y)|} \frac{(\eta(x)-\eta(y))K(x,y)}{|x-y|^s}dxdy \leq \int_{\Omega}f(x)\eta(x)dx
\end{align}
for all non-negative $\eta\in C^{\infty}_c(\Omega^{\prime}).$
\end{lemma}
\begin{proof}
Multiplying the above inequality \eqref{inq2.6} with $\eta(x)$ and integrating over $\Omega^{\prime}$ we get
\begin{align} \label{inq1}
    \int_{\Omega^{\prime}} \int_{\mathbb{R}^n} g\left(\frac{|u(x)-u(y)|}{|x-y|^s}\right) \frac{u(x)-u(y)}{|u(x)-u(y)|}\frac{\eta(x)}{|x-y|^s} K(x,y)\frac{dydx}{|x-y|^s}\leq \int_{\Omega}f(x)\eta(x)dx
\end{align}
Since $Supp(\eta)$ lies in $\Omega^{\prime},$ the outer integral of the above inequality can be replaced by $\mathbb{R}^n$ and a change of variable gives
\begin{align}\label{inq2}
    -\int_{\mathbb{R}^n} \int_{\mathbb{R}^n} g\left(\frac{|u(x)-u(y)|}{|x-y|^s}\right) \frac{u(x)-u(y)}{|u(x)-u(y)|}\frac{\eta(y)}{|x-y|^s} K(x,y)\frac{dxdy}{|x-y|^s}\leq \int_{\Omega}f(y)\eta(y)dy.
\end{align}
Adding the inequalities \eqref{inq1}-\eqref{inq2} we get \eqref{inq3}.
\end{proof}

Now we prove a logarithmic estimate similar to \cite{byun2022}.
\begin{proposition}\label{Loglemma}
Let $\Omega$ be a bounded domain and $c\in L^1_{loc}(\Omega).$ If $u \in W^{s,G}_{loc}(\Omega) \cap L^g_s(\mathbb{R}^n)$ is a weak supersolution of $\eqref{nonlocal equation}$ and $u \geq 0$ in $B_R(x_0) \subset \Omega,$ then for any $0<\alpha<1$ and $0<r <\frac{R}{2}, $ we have the following estimate
\begin{align}\label{loglemma}
    \int_{B_r} \int_{B_r} \left|\frac{log(u(x)+\alpha)}{log(u(y)+\alpha)}\right| \frac{dxdy}{|x-y|^n} \leq Cr^n\left\{1+ \frac{r^s}{g(\frac{\alpha}{r^s})} \textnormal{Tail}(u_{-}, x_0, R)\right\} + K_0 ||c||_{L^1(B_R)}
\end{align}
for some $C, K_0$ depending on $n,s,p,q, \lambda, \Lambda, diam(\Omega).$
\end{proposition}
\begin{proof}
Fix a test function $\phi \in C^{\infty}_c(B_{\frac{3r}{2}})$ such that
\[0\leq \phi \leq 1,\,\,\,\, \phi\equiv1 \,\,\, \text{in}\,\,\, B_r\,\,\,\, \text{and}\,\,\,\, \left|D\phi\right| \leq \frac{4}{r}.\]
Since $\eta:=\frac{(u+\alpha) \phi^{q}}{G(\frac{u+\alpha}{r^s})} \in W^{s, G}_0(\Omega)$ for $q\geq 1,$ $\eta \geq 0$ and $u$ is a supersolution of $\eqref{nonlocal equation},$ taking $\eta$ as a test function we get
\begin{align}\label{eq2.5}
    &\int_{B_{\frac{3r}{2}}}c(x)g(u(x)) \frac{(u(x)+\alpha) \phi^q(x)}{G(\frac{(u(x)+\alpha)}{r^s})}dx\nonumber \\
    &\leq\int_{B_{2r}}\int_{B_{2r}} g\left(\frac{|u(x)-u(y)|}{|x-y|^s}\right) \frac{u(x)-u(y)}{|u(x)-u(y)|} \frac{(\eta(x)-\eta(y))K(x,y)}{|x-y|^s}dxdy \nonumber\\
    &+2\int_{\mathbb{R}^n \setminus B_{2r}}\int_{B_{2r}}g\left(\frac{|u(x)-u(y)|}{|x-y|^s}\right) \frac{u(x)-u(y)}{|u(x)-u(y)|} \frac{\eta(x)K(x,y)}{|x-y|^s}dxdy :=I_1+I_2
\end{align}
Following the proof of \cite[Proposition~3.4]{byun2022}, we have
\begin{align*}
    I_1 \leq -\lambda \tilde{C} \int_{B_r}\int_{B_r} \left|\frac{log\left(u(x)+\alpha\right)}{log\left(u(y)+\alpha\right)}\right| \frac{dxdy}{|x-y|^n} + Cr^n
\end{align*}
and 
\begin{align*}
    I_2 \leq Cr^n \Big[1+\frac{r^s}{g(\frac{\alpha}{r^s})} \textnormal{Tail}(u_{-}, x_0, R)\Big].
\end{align*}
Hence from the inequality \eqref{eq2.5}, we obtain
\begin{align}\label{eq2.6}
\int_{B_r}\int_{B_r} \left|\frac{log\left(u(x)+\alpha\right)}{log\left(u(y)+\alpha\right)}\right| \frac{dxdy}{|x-y|^n}\leq &\frac{1}{\lambda \tilde{C}}\int_{B_{\frac{3r}{2}}}|c(x)|g(u(x)) \frac{(u(x)+\alpha) \phi^q(x)}{G(\frac{(u(x)+\alpha)}{r^s})}dx\nonumber \\
&+\frac{Cr^n}{\lambda\tilde{C}}\Big[1+\frac{r^s}{g(\frac{\alpha}{r^s})} \textnormal{Tail}(u_{-}, x_0, R)\Big].
\end{align}
Now using the property of $G,$ we have
\begin{align}\label{con g}
     \min\left\{\left(\frac{1}{r^s}\right)^p, \left(\frac{1}{r^s}\right)^q\right\}G(u(x)+\alpha) \leq G\left(\frac{u(x)+\alpha}{r^s}\right) \leq \max\left\{\left(\frac{1}{r^s}\right)^p, \left(\frac{1}{r^s}\right)^q\right\} G(u(x)+\alpha).
\end{align}
Applying \eqref{con g} and the increasing property of $g$ into \eqref{eq2.6}, for $0<r< \frac{R}{2}$ we get
\begin{align}\label{eq2.8}
   g(u(x)) \frac{(u(x)+\alpha) \phi^q(x)}{G(\frac{(u(x)+\alpha)}{r^s})} \leq \frac{g(u(x)+\alpha) (u(x)+\alpha)}{G(u(x)+\alpha)}\cdot \frac{1}{\min\left\{\left(\frac{1}{r^s}\right)^p, \left(\frac{1}{r^s}\right)^q\right\} } \le C_1 q.
\end{align}
Plugging \eqref{eq2.8} in \eqref{eq2.6} we get the \eqref{loglemma} with constants $C, K_0$ depending on $n,p,q, \lambda, \Lambda, diam(\Omega).$
\end{proof}

\begin{lemma}\label{VWlemma}
Let $c(x)$ be a continuous function in $\overline{\Omega}.$ If $u \in W^{s,G}_{loc}(\Omega)\cap L^g_s(\mathbb{R}^n) \cap C(\overline{\Omega}) $ is a weak supersolution of \eqref{nonlocal equation} and $u \geq 0$ in $\Omega^c,$ then $u$ is a viscosity supersolution of \eqref{nonlocal equation}.
\end{lemma}
\begin{proof}
By the assumptions on $u,$ one can easily observe that the conditions  (1), (2), (4) in the \cref{defn2.2} are satisfied. To check 3, we proceed by contradiction. On the contrary, suppose there exists $x_0 \in \Omega$ and $\psi \in C^2(B_r(x_0)) \cap L^{\infty}(\mathbb{R}^n)$ such that
$(a)\,\, u(x_0)=\psi(x_0),$ $(b)$ either $(i)$ or $(ii)$ in 3 holds, and $(c)$ \[\mathcal{L} \psi(x_0) < c(x_0)g(|u(x_0)|)\frac{u(x_0)}{|u(x_0)|}.\]
Then by the continuity of $\mathcal{L}\psi$ (cf. \cref{Lemma 2.3}), there exists a $\kappa\in (0, r)$ such that for all $x \in B_{\kappa}(x_0),$ and for all $n > N_0,$ 
we have
\[\mathcal{L} \psi(x) + \frac{1}{n} \leq c(x)g(|u(x)|)\frac{u(x)}{|u(x)|}.\] Choosing $\epsilon=\frac{1}{2n},$ from \cref{regular peturbation} we conclude that there exist $\tilde{\theta}>0, \rho\in (0, \frac{\kappa}{2})$ and $\eta \in C^2_c(B_{\frac{\rho}{2}}(x_0))$ with $0\leq \eta \leq 1$ and $\eta(x_0)=1 $ such that $\psi_{\theta}(x)=\psi(x)+\theta \eta(x)$  satisfies 
\begin{equation*}
    \sup_{B_{\rho}(x_0)}\left|\mathcal{L}\psi(x)-\mathcal{L}\psi_{\theta}(x)\right| < \epsilon < \inf_{B_{\rho/2}(x_0)}\left(c(x)g(|u(x)|)\frac{u(x)}{|u(x)|}-\mathcal{L}\psi(x)\right)\,\,\,\, \text{for all}\,\,\,\, 0\leq \theta < \tilde{\theta}.
    \end{equation*}
Consequently, we have 
\[\sup_{B_{\rho}(x_0)}\left|\mathcal{L}\psi_{\theta}(x)\right|<c(x)g(|u(x)|)\frac{u(x)}{|u(x)|}.\]
Moreover, we have $\psi_{\theta} \leq u$ in $B^c_{\rho}(x_0)$ and using \cref{lemma2.4} with the definition of supersolution we obtain
\begin{equation*}
    \begin{aligned}
    &\int_{\mathbb{R}^n} \int_{\mathbb{R}^n} g\left(\frac{|\psi_{\theta}(x)-\psi_{\theta}(y)|}{|x-y|^s}\right) \frac{\psi_{\theta}(x)-\psi_{\theta}(y)}{|\psi_{\theta}(x)-\psi_{\theta}(y)|} \frac{(\eta(x)-\eta(y))K(x,y)}{|x-y|^s}dxdy\\
    &\leq \int_{\mathbb{R}^n} \int_{\mathbb{R}^n} g\left(\frac{|u(x)-u(y)|}{|x-y|^s}\right) \frac{u(x)-u(y)}{|u(x)-u(y)|} \frac{(\eta(x)-\eta(y))K(x,y)}{|x-y|^s}dxdy
    \end{aligned}
\end{equation*}
for all $\eta \in W_0^{s,G}(B_{\rho}(x_0)), \,\, \eta\geq 0.$
Hence by the comparison principle (cf. \cref{prop2.1}), we get $\psi_{\theta} \leq u$ in $B_{\rho}(x_0).$ But this leads to a contradiction since $u(x_0)< \psi_{\theta}(x_0)=\psi(x_0)+\theta \leq u(x_0).$  This completes the proof. 
\end{proof}
\section{Strong minimum principle} \label{sec3}
The Aim of this section is to prove \cref{mainth1}. We start with some preparatory lemmas.
\begin{lemma}\label{lemma3.1}
Let $c\in L^1_{loc}(\Omega),$ $c(x)\leq 0$ and $u$ be a weak supersolution of \eqref{nonlocal equation}. Then we have the following conclusions:
\begin{enumerate}
    \item If $u \geq 0\,\,\, a.e.$ in $\mathbb{R}^n \setminus \Omega$ and $u=0\,\,\, a.e.$ in $\Omega,$ then $u=0\,\,\, a.e.$ in all of $\mathbb{R}^n.$
    \item If $u \geq 0 \,\,\, a.e.$ in $\mathbb{R}^n \setminus \Omega$ for a bounded domain $\Omega,$ then $u \geq 0 \,\,\, a.e.$ in $\Omega.$
\end{enumerate}
\end{lemma}
\begin{proof}
Take any non negative test function $\eta \in C^{\infty}_c(\Omega).$ Since $u=0\,\,\, a.e.$ in $\Omega$ and $\eta=0$ in $\mathbb{R}^n \setminus \Omega,$ from the definition of weak solution, we have
\begin{align*}
&\int_{\mathbb{R}^n} \int_{\mathbb{R}^n} g\left(\frac{|u(x)-u(y)|}{|x-y|^s}\right) \frac{u(x)-u(y)}{|u(x)-u(y)|} \frac{(\eta(x)-\eta(y))K(x,y)}{|x-y|^s}dxdy\\
&= -2 \int_{\Omega} \int_{\mathbb{R}^n \setminus \Omega} g\left(\frac{\left|u(x)\right|}{\left|x-y\right|^s}\right)\frac{u(x)}{\left|u(x)\right|}\frac{\eta(y)}{\left|x-y\right|^s}K(x,y)dxdy \geq 0.
\end{align*}
Using the fact that $u\geq 0\,\,\, a.e.$ in $\mathbb{R}^n \setminus \Omega$ we get $u=0\,\,\, a.e.$ in $\mathbb{R}^n \setminus \Omega.$ Thus combining $u=0\,\,\, a.e.$ in $\Omega,$ we have $u=0\,\,\, a.e$ in $\mathbb{R}^n$. This proves $(1).$
Note that, since $u\in W^{s, G}_{loc}(\Omega)$ and $u\geq 0\,\,\, a.e$ in $\mathbb{R}^n\setminus\Omega,$ we get $u_{-} \in W^{s, G}_0(\Omega).$ Then using $u_{-}$ as a test function by definition of supersolution, we have
\begin{align*}
&\int_{\mathbb{R}^n} \int_{\mathbb{R}^n} g\left(\frac{|u(x)-u(y)|}{|x-y|^s}\right) \frac{u(x)-u(y)}{|u(x)-u(y)|} \frac{(u_{-}(x)-u_{-}(y))K(x,y)}{|x-y|^s}dxdy \\
&\geq \int_{\Omega} c(x)g(|u|)\frac{u}{|u|} u_{-}(x)dx \geq 0. 
\end{align*}
In the last inequality, we used $c(x) \leq 0$ and $\frac{u(x) u_{-}(x)}{|u(x)|}$ is always non-positive for any $x \in \Omega$. On the other hand, noting that the term $(u(x)-u(y))(u_{-}(x)-u_{-}(y)) \leq 0$ for any $x,y \in \Omega$, we get
\begin{align*}
   \int_{\mathbb{R}^n} \int_{\mathbb{R}^n} g\left(\frac{|u(x)-u(y)|}{|x-y|^s}\right) \frac{u(x)-u(y)}{|u(x)-u(y)|} \frac{(u_{-}(x)-u_{-}(y))K(x,y)}{|x-y|^s}dxdy\leq 0. 
\end{align*}
Since $u_{-}$ is zero outside of $\Omega$ and the above inequalities show that $u_{-}$ is constant inside $\Omega,$ using \cite[Proposition~2.9]{Salortpoyla-szeo} we conclude that $u_{-}$ is identically zero in $\mathbb{R}^n.$ This proves the second part of the lemma.
\end{proof}
\begin{lemma}\label{lemma3.2}
Let $c\in C(\overline{B_R(x_0)})$ be non-positive and $u \in W^{s, G}_{loc}(B_R(x_0))\cap L^g_s(\mathbb{R}^n) \cap C(\overline{B_R(x_0)})$ be a weak supersolution of \eqref{nonlocal equation} in $B_R(x_0).$ If $u \geq 0$ in $B^c_R(x_0)$, then either $u>0$ in $B_R(x_0)$ or $u=0\,\, a.e.$ in $\mathbb{R}^n.$
\end{lemma}
\begin{proof}
Our aim is to show that if $u(\tilde{x})=0$ for some $\tilde{x} \in B_R(x_0),$ then $u=0\,\, a.e.$ in $\mathbb{R}^n.$ First, using \cref{lemma3.1} (2), we get $u \geq 0\,\, a.e$ in $B_R(x_0).$ Moreover, by the equivalence of viscosity solution and weak solution (\cref{VWlemma}), we say that $u$ is a viscosity solution of \eqref{nonlocal equation} in $B_R(x_0).$ Now for any $\epsilon >0,$ we choose a test function 
\[\psi^{\epsilon}(x)=-\epsilon \left|x-\tilde{x}\right|^{\beta},\]
where $\beta > \max\left\{2, \frac{sp}{p-1}\right\}.$ Therefore by the definition of viscosity supersolution,
\begin{align}\label{equation3.1}
\mathcal{L}\psi^{\epsilon}_r(\tilde{x})\geq c(\tilde{x})g(|u(\tilde{x})|)\frac{u(\tilde{x})}{|u(\tilde{x})|}=0,
\end{align}
where $\psi_r$ is defined as
\begin{align*}
    \psi^{\epsilon}_r(x)=\begin{cases}
\psi^{\epsilon}(x)\,\,\,\, \textnormal{in}\,\,\, B_r(\tilde{x})\\
u(x)\,\,\,\, \textnormal{otherwise}
\end{cases}
\end{align*}
for some $r \in (0, R-|x_0-\tilde{x}|).$ From \eqref{equation3.1}, we obtain
\begin{align}\label{equ3.5}
    \int_{B_r(\tilde{x})} g\left(\frac{\epsilon|y-\tilde{x}|^{\beta}}{|\tilde{x}-y|^s}\right)K(\tilde{x},y)\frac{dy}{|\tilde{x}-y|^s} \geq \int_{B^c_r(\tilde{x})} g\left(\frac{|u(y)|}{|\tilde{x}-y|^s}\right)\frac{u(y)}{|u(y)|} K(\tilde{x},y)\frac{dy}{|\tilde{x}-y|^s}.
\end{align}
Observe that when $p\leq \frac{2}{2-s},$ $\beta > \frac{sp}{p-1}$ and using \eqref{condition on g} with the property $g(\epsilon t) \leq \epsilon g(t) $ for $\epsilon <1,$ we obtain
\begin{align}\label{eq3.6}
   & \int_{B_r(\tilde{x})} g\left(\frac{\epsilon|y-\tilde{x}|^{\beta}}{|\tilde{x}-y|^s}\right)K(\tilde{x},y)\frac{dy}{|\tilde{x}-y|^s}\nonumber \\
    &\leq \epsilon \Lambda \int_{B_r(\tilde{x})}g\left(|y-\tilde{x}|^{\beta-s}\right)\frac{dy}{|\tilde{x}-y|^{n+s}}= \epsilon \Lambda \int_{B_r(\tilde{x})}\left(\int_0^1\frac{d}{d\theta} g\left(\theta |\tilde{x}-y|^{\beta-s}\right)d\theta\right)\frac{dy}{|\tilde{x}-y|^{n+s}}\nonumber\\
    &= \epsilon \Lambda \int_{B_r(\tilde{x})}\int_0^1 g^{\prime}\left(\theta|\tilde{x}-y|^{\beta-s}\right)\cdot |\tilde{x}-y|^{(\beta-s)-(n+s)}d\theta dy \nonumber\\
    & \leq \epsilon \Lambda q_g \int_{B_r(\tilde{x})} \int_0^1 \frac{g\left(\theta |\tilde{x}-y|^{\beta-s}\right)}{\theta |\tilde{x}-y|^{\beta-s}} \cdot |\tilde{x}-y|^{(\beta-s)-(n+s)}d\theta dy \nonumber \\
    &\leq \epsilon \Lambda q_g\int_{B_r(\tilde{x})} \int_0^1 \frac{g(\theta)}{\theta}\cdot\left|\tilde{x}-y\right|^{(\beta-s)(p-1)-(n+s)} d\theta dy \nonumber \\
    &\leq \frac{\epsilon \Lambda q_g}{p_g} \int_{B_r(\tilde{x})} \int_0^1 g^{\prime}(\theta) \cdot \left|\tilde{x}-y\right|^{(\beta-s)(p-1)-(n+s)} d\theta dy \leq \frac{  \epsilon C_0 \Lambda q_g}{p_g} \int_{B_r(\tilde{x})}\left|\tilde{x}-y\right|^{(\beta-s)(p-1)-(n+s)}dy \nonumber \\
    &=\frac{\epsilon C_0 \Lambda q_g r^{(\beta-s)(p-1)-s}}{p_g (\beta-s)(p-1)-s} \to 0
    \,\,\, \text{as} \,\,\, \epsilon \to 0.
\end{align}
 When $p>\frac{2}{2-s}$ a similar calculation yields that the L.H.S of \eqref{equ3.5} goes to zero as $\epsilon \to 0.$ Hence R.H.S of \eqref{equ3.5} gives $u\leq 0$ in $B^c_r(\tilde{x}).$ Since $B^c_r(\tilde{x})=\left(B_R(x_0)\setminus B_r(x_0)\right) \cup B^c_R(x_0),$ by the assumption on $u$ together with \cref{lemma3.1} (2) , we conclude that $u=0\,\, a.e.$ in $B^c_r(\tilde{x}).$ Now proceeding as above the same can be shown for any $r^{\prime}<r,$ that is, $u=0\,\, a.e.$ in $B^c_{r^{\prime}}(\tilde{x}).$ Therefore $u=0\,\, a.e.$ in $\mathbb{R}^n.$
\end{proof}
\begin{remark}
Observe that if $\beta \geq 2,$ then $\psi_r^{\epsilon} \in C^2_{\beta}(B_r(x_0))$ and thus becomes a valid test function. However using the assumption on $u$ and the restriction on $\beta,$ \eqref{eq3.6} shows that $\psi_r^{\epsilon} \in L^g_s(\mathbb{R}^n)$ for any fixed $\epsilon >0.$ 
\end{remark}
Now we are in a position to prove \cref{mainth1}.
\begin{proof}[Proof of \cref{mainth1}]
The proof is divided into two cases. In the first case we consider $c\in L^1_{loc}(\Omega)$ and  $c, u \in C(\overline{\Omega})$ is treated in the second case.\\
{\bf Case 1: $ c \in L^1_{loc}(\Omega).$} We provide a proof for this case in three steps.\\
{\it Step 1:} In the first step we aim to prove if $\Omega$ is bounded and connected, and $u\geq 0\,\, a.e.$ in $\Omega^c,$ then either $u>0\,\, a.e$ in $\Omega$ or $u=0\,\, a.e$ in $\mathbb{R}^n.$ By \cref{lemma3.1} we have $u\geq 0\,\, a.e.$ in $\mathbb{R}^n.$ Following \cite[Theorem A.1]{brasco} we prove that if $u\not\equiv 0$ in $\Omega$ then $u>0\,\, a.e.$ in $\Omega.$ We first prove the assertion for every compact set $K\subset \Omega.$ Then connectedness of $\Omega$ implies that there exists a sequence of compact connected set $K_n \subset \Omega$ such that $u>0$ in $K_n,$ for all $n$ and $\left|\Omega\setminus K_n\right| <\frac{1}{n}.$ Finally letting $n\to \infty,$ we conclude the result. Let $K$ be compact connected and $K:=\{x \in \Omega | dist\left(x, \partial \Omega\right)>2r\}$ for some $r>0$ such that $u\neq 0\,\, a.e.$ in $K.$ Then we claim that $u>0\,\, a.e.$ in $K.$ Since $K$ is compact, $K \subset \cup_{i=1}^m B_{r/2}(x_i)$ with $\left|B_{r/2}(x_i)\cap B_{r/2}(x_{i+1})\right|>0$ for $i=1,..,m-1.$ Now, on the contrary, assume that the zero set of $u,$ $Z=\left\{x \in B_{r/2}(x_i)| u(x)=0\right\}$ for some $i \in \{1,..,m-1\},$ has a positive measure. For every $\alpha>0,$ consider the function 
\[F_{\alpha}(x)=log\left(1+ \frac{u(x)}{\alpha}\right)\,\,\,\, \textnormal{for}\,\,\,\, x\in B_{r/2}(x_i).\]
We see that $F_{\alpha}(x)=0$ for all $x\in Z$ and for every $x\in B_{r/2}(x_i)$ and $y\neq x,$ $y \in Z$ we get
\begin{align}\label{equ3.4}
|F_{\alpha}(x)|=\frac{|F_{\alpha}(x)-F_{\alpha}(y)|}{|x-y|^n} \cdot |x-y|^n.
\end{align}
Integrating \eqref{equ3.4} over $Z$ and $B_{r/2}(x_i),$ we get
\begin{align*}
\int_{B_{r/2}(x_i)} \left|F_{\alpha}(x)\right|dx&\leq \frac{\left(\displaystyle{\sup_{x,y \in B_{r/2}(x_i)}}\left|x-y\right|^n\right)}{|Z|}\int_{B_{r/2}(x_i)} \int_{B_{r/2}(x_i)}\frac{|F_{\alpha}(x)-F_{\alpha}(y)|}{|x-y|^n}dxdy\\
&\leq \frac{Cr^n}{|Z|}\int_{B_{r/2}(x_i)} \int_{B_{r/2}(x_i)}\frac{|F_{\alpha}(x)-F_{\alpha}(y)|}{|x-y|^n}dxdy\\
&=\frac{Cr^n}{|Z|}\int_{B_{r/2}(x_i)} \int_{B_{r/2}(x_i)}\left|\frac{log(u(x)+\alpha)}{log(u(y)+\alpha)}\right| \frac{dxdy}{|x-y|^n}.
\end{align*}
Now using \cref{Loglemma} and $u\geq 0\,\, a.e.$ in $\mathbb{R}^n$ we obtain that 
\begin{align}\label{equ3.6}
\int_{B_{r/2}(x_i)}\left|log\left(1+\frac{u(x)}{\alpha}\right)\right|dx \leq \frac{Cr^{2n}}{|Z|}+K_0 r^n||c||_{L^1(B_r)},
\end{align}
where $C$ and $K_0$ are independent of $\alpha.$ Now letting $\alpha \to 0$ in \eqref{equ3.6} we conclude that $u=0\,\,a.e.$ in $B_{r/2}(x_i).$ Repeating the same argument for other balls that cover $K,$ we infer that $u=0\,\, a.e.$ in $K$ which is contradiction. Thus $u>0\,\, a.e.$ in $K.$\\
{\it Step 2:} If $\Omega$ is unbounded and connected, we follow the arguments of \cite[Lemma~3.3]{delpezzoJDE2017} to conclude $u>0\,\,a.e.$ in $\Omega.$\\
{\it Step 3:} Note that the proof of {\bf Case 1} can be completed if we can remove the connectedness condition on $\Omega$ from the previous steps combining with \cref{lemma3.1}. Therefore it is enough to show that if $u\not\equiv 0$ in $\Omega$ then $u\not\equiv 0$ in all the connected components of $\Omega.$ Again, this follows from the arguments of \cite[Lemma~3.4]{delpezzoJDE2017}.\\
{\bf Case 2: $ c, u \in C(\bar{\Omega}).$} For any open set $\Omega \subset \mathbb{R}^n,$ let $x_0\in \Omega$ such that $u(x_0)=0.$ Consider the ball $B_R(x_0)\subset \Omega.$ Since $c, u \in C(\Omega)$ then $c, u \in C(\overline{B_R(x_0)})$ and also $c \in L^1_{loc}(\Omega).$ Suppose $\Omega$ is bounded and $u \geq 0$ in $\Omega^c.$ By using the second assertion of \cref{lemma3.1} we conclude that $u\geq 0$ in $B^c_R(x_0).$ Now since $u(x_0)=0,$ by \cref{lemma3.2} we have $u=0\,\,a.e.$ in $\mathbb{R}^n.$ If $u \geq 0$ in $\mathbb{R}^n,$ then again using the same argument we conclude that either $u>0$ in $\Omega$ or $u=0\,\, a.e.$ in $\mathbb{R}^n.$ This completes the proof of the theorem.
\end{proof}
\section{Hopf's lemma} \label{sec4}
In this section, we shall prove \cref{mainth2}. We will argue as in \cite[Lemma~4.1]{delpezzoJDE2017}. Since $\Omega$ satisfies interior ball condition, it is enough to prove \cref{mainth2} for a ball, i.e, $\Omega= B_r(x_0).$ For any general domain $\Omega,$ from the interior ball condition we deduce for any $y\in \partial \Omega,$ there exists a ball $B_r(x_0) \subset \Omega$ such that $y \in \partial B_r$ and the proof follows.
\begin{proof}[Proof of \cref{mainth2}]
First we make an observation about non-positivity of $c(x).$ In \cref{mainth1}, when $c,u \in C(\overline{\Omega})$ and $c(x) \leq 0,$ for a bounded domain $\Omega$ with $u\geq 0$ in $\Omega^c$ we conclude that either $u>0$ in $\Omega$ or $u=0\,\, a.e.$ in $\mathbb{R}^n.$ The same conclusion holds if $u \geq 0\,\, a.e.$ in $\mathbb{R}^n.$ In the second part of \cref{mainth2} no non-positivity assumption is made on $c(x).$ If $u \geq 0\,\,a.e$ in $\mathbb{R}^n,$ since $c(x) \geq -c_{-}(x),$ $u$ is a weak supersolution of
\[\mathcal{L}u=-c_{-}(x)g(u)\,\,\,\, \textnormal{in}\,\,\,\, \Omega. \] Hence by \cref{mainth1} ({\bf Case 2}) either $u>0$ in $\Omega,$ in particular $\Omega=B_r(x_0)$ or $u=0,\,\, a.e.$ in $\mathbb{R}^n.$ Suppose $u\not\equiv 0 $ in $B_r(x_0),$ and we need to show \eqref{Hlemma} holds for all $y \in \partial B_r.$ Using \cref{salortNAlemma} and \cref{prop2.2}, for any compact set $K \subset B_r \setminus B_{\rho},$ and $d>0$ we obtain that $\delta^s + d\chi_K$ is a weak solution of 
\[\mathcal{L}(\delta^s+d\chi_K)=f+h_d\,\,\,\, \textnormal{in}\,\,\,\, B_{\rho},\] where
\begin{align*}
h_d(x):=2 \int_{K}\Big[&g\left(\frac{\left|\delta^s(x)-\delta^s(y)-d\right|}{\left|x-y\right|^s}\right)\frac{\delta^s(x)-\delta^s(y)-d}{\left|\delta^s(x)-\delta^s(y)-d\right|}\\
&-g\left(\frac{\left|\delta^s(x)-\delta^s(y)\right|}{\left|x-y\right|^s}\right)\frac{\delta^s(x)-\delta^s(y)}{\left|\delta^s(x)-\delta^s(y)\right|}\Big] K(x,y)\frac{dy}{\left|x-y\right|^s}\\
=2 \int_{K}\Big[&\frac{g\left(\frac{\left|\delta^s(x)-\delta^s(y)-d\right|}{\left|x-y\right|^s}\right)}{\frac{\left|\delta^s(x)-\delta^s(y)-d\right|}{|x-y|^s}}\cdot\frac{\delta^s(x)-\delta^s(y)-d}{\left|x-y\right|^s}\\
&-g\left(\frac{\left|\delta^s(x)-\delta^s(y)\right|}{\left|x-y\right|^s}\right)\frac{\delta^s(x)-\delta^s(y)}{\left|\delta^s(x)-\delta^s(y)\right|}\Big] K(x,y)\frac{dy}{\left|x-y\right|^s}\\
\leq 2\int_{K}\Big[&\frac{1}{q_g}g^{\prime}\left(\frac{\left|\delta^s(x)-\delta^s(y)-d\right|}{\left|x-y\right|^s}\right)\frac{\delta^s(x)-\delta^s(y)-d}{\left|x-y\right|^s}\\
&-g\left(\frac{\left|\delta^s(x)-\delta^s(y)\right|}{\left|x-y\right|^s}\right)\frac{\delta^s(x)-\delta^s(y)}{\left|\delta^s(x)-\delta^s(y)\right|}\Big] K(x,y)\frac{dy}{\left|x-y\right|^s}.
\end{align*}
Note that in the last inequality we used \eqref{condition on g} and $\frac{\delta^s(x)-\delta^s(y)-d}{\left|x-y\right|^s} \leq 0.$ Since $g$ and $g^{\prime}$ are continuous and $K$ is compact
\begin{align*}
h_d(x)
&\leq 2\left(||g^{\prime}||_{L^{\infty}(K)} +||g||_{L^{\infty}(K)}\right)\int_K \Big[\frac{\delta^s(x)-\delta^s(y)-d}{q_g\left|x-y\right|^s}+\frac{\delta^s(x)-\delta^s(y)}{\left|\delta^s(x)-\delta^s(y)\right|}\Big] K(x,y)\frac{dy}{\left|x-y\right|^s}\\
&\to -\infty\,\,\,\, \textnormal{in}\,\,\, B_{\rho}\,\,\,\, \textnormal{as} \,\,\,\, d \to \infty.
\end{align*}
In the previous conclusion we also make use of $\delta^s\in L^{\infty}(\Omega)$ and $dist\left(K, B_{\rho}\right)>0.$ Now we can choose a large $d$ such that
\begin{align}\label{equ4.1}
\sup_{B_{\rho}} f(x)+h_d(x) \leq \inf_{B_{\rho}} c(x)g(u(x)).
\end{align}
Furthermore, for a small $\kappa>0,$ we have $\kappa(r^s+d) \leq \displaystyle{\inf_{B_r\setminus B_{\rho}}} u(x)$ and thus we find that $\kappa (\delta^s+d\chi_K) \leq u$ in $B^c_{\rho}.$ Now using \eqref{equ4.1} and the definition of weak supersolution, we get
\begin{align*}
&\int_{\mathbb{R}^n} \int_{\mathbb{R}^n} g\left(\frac{|u(x)-u(y)|}{|x-y|^s}\right)\frac{u(x)-u(y)}{|u(x)-u(y)|}(\eta(x)-\eta(y))K(x,y)\frac{dxdy}{|x-y|^s} \\
&\geq
\int_{\mathbb{R}^n} \int_{\mathbb{R}^n} g\left(\frac{|\kappa v(x)-\kappa v(y)|}{|x-y|^s}\right)\frac{ v(x)- v(y)}{| v(x)- v(y)|}(\eta(x)-\eta(y))K(x,y)\frac{dxdy}{|x-y|^s} 
\end{align*}
where $v=\delta^s+d\chi_K.$ Therefore by using comparison principle (cf. \cref{prop2.1}) we obtain $\kappa v \leq u$ in $B_{\rho}.$ Hence
\[\kappa \leq \frac{u(x)}{\delta^s(x)}\,\,\, \text{for all}\,\,\, x \in B_{\rho}.\] Thus we find
\[\liminf_{B_r(x_0) \ni x \to y } \frac{u(x)}{\delta^s(x)} >0, \,\,\, \text{for all}\,\,\, y\in \partial B_r.\] This completes the proof.
\end{proof}
\subsection*{Acknowledgement.}The author is indebted to Anup Biswas and Saibal Khan for helpful discussions and suggestions. The author also thanks the referee for his/her careful reading of the manuscript and valuable suggestions.

\end{document}